\documentclass[12pt]{article}
\usepackage{graphicx,amsmath,amsthm,amssymb, color}
\input amssym.def
\input amssym.tex
\date{}
\newtheorem{lemma}{Lemma}

\newtheorem{theorem}[lemma]{Theorem}

\begin{document}

\title{{\bf A short proof of some recent results related to Ces{\`a}ro function spaces}}

\author {Sergey V. Astashkin\thanks{Research partially supported by
RFBR grant no. 12-01-00198-a.} \,{\small and} Lech Maligranda}

\date{}

\maketitle

\renewcommand{\thefootnote}{\fnsymbol{footnote}}


\vspace{-7 mm}

\begin{abstract}
\noindent {\footnotesize We give a short proof of the recent results that, for every $1\leq p< \infty,$ 
the Ces{\`a}ro function space $Ces_p(I)$ is not a dual space, has the weak Banach-Saks property 
and  does not have the Radon-Nikodym property.}
\end{abstract}


\medskip

The main purpose of this paper is to give a short proof of the following recent results related to 
the Ces{\`a}ro function spaces: for every $1\le p< \infty$ the space $Ces_{p} = Ces_{p}(I)$ has 
the weak Banach-Saks property \cite[Theorem 8]{AM}, does not have the Radon-Nikodym 
property and it is not a dual space \cite[Corollaries 5.1 and 5.5]{KamKub}.

{\it The Ces{\`a}ro function spaces} $Ces_{p} = Ces_{p}(I)$ $(1 \leq p < \infty)$, where $I = [0, 1]$ 
or $I = [0, \infty),$ are the classes of all Lebesgue measurable real functions $f$ on $I$ such that
$$
\|f\|_{C(p)} = \left [ \int_I \left (\frac{1}{x}\int_{0}^{x} |f(t)|
~dt \right)^{p} ~dx \right]^{1/p} < \infty.
$$

A Banach space $X$ is said to have the {\it weak Banach-Saks property} if every weakly null 
sequence in $X$, say $(x_n)$, contains a subsequence $(x_{n_k})$ whose first arithmetical 
means converge strongly to zero, that is, $\lim_{m\to \infty} \frac{1}{m}\Big\|\sum_{k=1}^m 
x_{n_k}\Big\|_X = 0.$

It is known that uniformly convex spaces, $c_0$, $l^1$ and $L^1$ have the weak Banach-Saks 
property, whereas $C[0,1]$ and $l^\infty$ do not have. We should mention that the result on 
$L^1$ space, proved by Szlenk \cite{Sz} in 1965, was a very important break-through in 
studying of the weak Banach-Saks property.

In 1982, Rakov \cite[Theorem 1]{Ra} proved that a Banach space with non-trivial type (or equivalently
B-convex) has the weak Banach-Saks property (cf. also Tokarev \cite[Theorem 1]{To}). Recently 
Dodds-Semenov-Sukochev \cite{DSS} investigated the weak Banach-Saks property of rearrangement 
invariant spaces and Astashkin-Sukochev \cite{AS} have got a complete description of Marcinkiewicz 
spaces with the latter property.

The spaces $Ces_{p}[0, 1]$ for $1 \leq p < \infty$ are neither B-convex (they have trivial type) nor 
rearrangement invariant. Nevertheless, by studying the dual space $Ces_{p}[0, 1]^*,$ 
Astashkin-Maligranda \cite[Theorem 8]{AM} proved that these spaces have the weak Banach-Saks 
property. Here, we present another simpler proof of this result which does not use any knowledge of the 
structure of the latter dual space. 

\vspace{3mm}
\begin{theorem}\label{THEOREM 1.}
 For every $1 \leq p < \infty$ the Ces{\`a}ro function space $Ces_{p}(I)$ has the weak Banach-Saks 
 property.
\end{theorem}

The proof will be based on the following simple observation. Recall that {\it the space with mixed norm} 
$L^p(I)[L^1[0, 1]]$ consists of all classes of Lebesgue measurable functions on $I\times [0,1]$
 $x(s,t)$ such that for a.e. $s\in I$ the function $x(s,\cdot)\in L^1[0, 1]$ and the function 
 $\|x(s,\cdot)\|_{L^1[0, 1]}\in L^p(I)$ with the norm $ \|x\|_{L^p(I)[L^1[0, 1]]}=\|\|x(s,\cdot)\|_{L^1[0, 1]}\|_{L^p(I)}$ 
(see, for example, \cite[\S\,11.1, p.~400]{KA}).

\begin{lemma}\label{LEMMA 1.}
For every $1 \leq p < \infty$ the space $Ces_{p}(I)$ is isometric to a closed subspace of the mixed norm 
space $L^p(I)[L^1[0, 1]]$.
\end{lemma}
\begin{proof} In fact, the mapping $f(t) \longmapsto {\cal S}f(x,t) = f(xt)$ is such an isometry from
$Ces_{p}(I)$ into $L^p(I)[L^1[0, 1]]$ since
$$
\| f\|_{C(p)} = \Big\| \frac{1}{x} \int_0^x |f(t)| \,dt \Big\|_{L^p(I)} = \Big\| \int_0^1 |f(tx)| \,dt \Big\|_{L^p(I)} = 
\left \|\, \| {\cal S}f(x,\cdot)\|_{L^1[0,1]} \right \|_{L^p(I)}.
$$ 
\end{proof}

\begin{proof}[Proof of Theorem \ref{THEOREM 1.}]
Firstly, we note that the Bochner vector-valued Banach space \newline$L^p(I, L^1[0, 1])$ coincides with the 
mixed norm space $L^p(I)[L^1[0, 1]]$ (see \cite[Theorem 1.1]{El58}, \cite[Theorem 2.2]{Bu}; cf. also 
\cite[pp. 282-283]{Mal}). Moreover, by the Szlenk theorem \cite{Sz}, the space 
$L^1(I)[L^1[0, 1]]=L^1(I\times [0,1])$ has the weak Banach-Saks property. Therefore, applying the Cembranos 
theorem \cite[Theorem C]{Ce94} (see also \cite[pp. 295-302]{Lin04}), we see that the same is true also for the 
space $L^p(I)[L^1[0, 1]].$ Since, due to Lemma \ref{LEMMA 1.}, the Ces{\`a}ro function space $Ces_{p}(I)$ 
is isometric to a closed subspace of $L^p(I)[L^1[0, 1]]$ and any closed subspace inherits the weak Banach-Saks 
property, then $Ces_{p}(I)$ has this property as well. The proof is complete.
\end{proof}

The following results were proved in \cite{KamKub} (see Corrollaries 5.1 and 5.5) by using an isometric 
representation of the dual space of $Ces_{p}(I)$, $1\le p<\infty.$ Here, we show that they are rather simple 
consequences of well-known classical theorems.

\begin{theorem}\label{THEOREM 2.}
Let $1 \leq p < \infty.$ Then

(a) $Ces_{p}(I)$ is not a dual space; 

(b) $Ces_{p}(I)$ does not have the Radon-Nikodym property.
\end{theorem}

Firstly, we prove the following auxiliary statement.

\begin{lemma}\label{LEMMA 2.}
For every $1 \leq p < \infty$ there is a norm $\|\cdot\|_{C(p)}^*$ equivalent to the usual norm in 
$Ces_{p}(I)$ such that the space $(Ces_{p}(I),\|\cdot\|_{C(p)}^*)$ contains a closed subspace 
isometric to the space $L^1[0,1].$
\end{lemma}
\begin{proof}
For arbitrary $f\in Ces_p:=Ces_{p}[0,1]$ we set
$$
\|f\|_{C(p)}^* := \left\|f\cdot\chi_{[0,1/4]\cup [3/4,1]}\right\|_{C(p)}+\|f\cdot\chi_{(1/4,3/4)}\|_{L^1}.
$$
Since
\begin{eqnarray*}
\left\|f\cdot\chi_{(1/4,3/4)}\right\|_{C(p)}^p &=&\int_{1/4}^{3/4}\left(\frac1x
\int_{1/4}^{x}|f(s)|\,ds\right)^p\,dx+\int_{3/4}^{1}\left(\frac1x
\int_{1/4}^{3/4}|f(s)|\,ds\right)^p\,dx\\
&\le&\left(\frac124^p+\frac14\Big(\frac43\Big)^p\right)\left(\int_{1/4}^{3/4}|f(s)|\,ds\right)^p\le
4^p\|f\cdot\chi_{(1/4,3/4)}\|_{L^1}^p,
\end{eqnarray*}
we have
$$
\|f\|_{C(p)}\le 4\|f\|_{C(p)}^*\;\;(f\in Ces_p).
$$
Conversely,
$$
\left\|f\cdot\chi_{(1/4,3/4)}\right\|_{C(p)}^p\ge 
\int_{3/4}^{1}\left(\frac1x \int_{1/4}^{3/4}|f(s)|\,ds\right)^p\,dx\ge
\frac14\|f\cdot\chi_{(1/4,3/4)}\|_{L^1}^p,
$$
whence for every $f\in Ces_p$
$$
\|f\|_{C(p)}^*\le \|f\|_{C(p)}+\|f\cdot\chi_{(1/4,3/4)}\|_{L^1}\le 5\|f\|_{C(p)}.
$$
Therefore, the norms $\|\cdot\|_{C(p)}^*$ and $\|\cdot\|_{C(p)}$ are equivalent on
$Ces_p:=Ces_{p}[0,1].$ Since the mapping 
$$
f(t)\longmapsto 
{\mathcal H} f(t): = \begin{cases}2f(2t-1/2)  &\text{if}\;\; 1/4<t<3/4\\
0 &\text{if}\;\; 0\le t\le 1/4 \;\;\text{or}\;\; 3/4\le t\le 1\end{cases}.
$$
is a linear isometry from $L^1[0,1]$ onto the subspace of $(Ces_{p}[0,1],\|\cdot\|_{C(p)}^*)$ consisting 
of all functions with support from the interval $(1/4,3/4),$ we obtain the result in the case $I=[0,1].$ 
If $I=[0,\infty)$ the proof follows in the same way.
\end{proof}

\begin{proof}[Proof of Theorem \ref{THEOREM 2.}]
Assume that $(a)$ is not valid and $Ces_p$ is a dual space. Since this property is invariant with respect 
isomorphisms, we see that $(Ces_{p}(I),\|\cdot\|_{C(p)}^*)$, where $\|\cdot\|_{C(p)}^*$ is the norm from
Lemma \ref{LEMMA 2.}, is also a dual space. Then, since it is separable, by the classical 
Bessaga-Pe{\l}czy\`nski result \cite{BP}, the space $(Ces_{p}(I),\|\cdot\|_{C(p)}^*)$ has the Krein-Milman 
property (i.e., every closed bounded set in this space is the closed convex hull of its extreme points). 
Since $(Ces_{p}(I),\|\cdot\|_{C(p)}^*)$ contains a closed subspace isometric to the space $L^1[0,1],$ the 
latter contradicts to the fact that the closed unit ball in $L^1[0,1]$ has no extreme points. Therefore, 
$(a)$ is proved. 

It is well-known that every Banach space which has the Radon-Nikodym property possesses also the 
Krein-Milman property \cite[Theorem 6.5.1]{Dis} (see also \cite[p. 118]{AK}, \cite[p. 229]{Pi} and the 
references given there). Thus, we obtain $(b)$, and the proof is complete. 
\end{proof}

\vspace{3mm}

\noindent
{\footnotesize Department of Mathematics and Mechanics, Samara State University\\
Acad. Pavlova 1, 443011 Samara, Russia} ~{\it E-mail address:} {\tt astash@samsu.ru} \\

\vspace{-3mm}

\noindent
{\footnotesize Department of Engineering Sciences and Mathematics, Lule\r{a} University of Technology\\
SE-971 87 Lule\r{a}, Sweden} ~{\it E-mail address:} {\tt lech.maligranda@ltu.se} \\

\end{document}